\documentclass{amsart}
\usepackage{amsmath,amssymb,amsthm,young,fullpage,graphicx,tikz}
\usepackage[boxsize=1em,centertableaux]{ytableau}
\usepackage{hyperref}
\newcommand{\arxiv}[1]{#1} \newcommand{\adv}[1]{}

\arxiv{\usepackage[alphabetic]{amsrefs} }

\theoremstyle{definition}
\newtheorem{theorem}{Theorem}
\newtheorem{lemma}{Lemma}

\newtheorem{prop}{Proposition}

\DeclareMathOperator{\sh}{sh}

\DeclareMathOperator{\sgn}{sgn}

\newcommand{\bu}[1]{\mathbf{\underline{#1}}}

\newcommand{\bY}{\begin{Young}}
\newcommand{\eY}{\end{Young}}

\title{ Tableaux and plane partitions of truncated shapes}
\author{ Greta Panova }
\thanks{The majority of this work was done in 2010 when the author was at Harvard University, Department of Mathematics. \\
Currently supported by a Simons Postdoctoral Fellowship at UCLA.} 
\curraddr{ Mathematics Department, University of California, Los Angeles, 520 Portola Plaza, Los Angeles, CA 90095.  } 
\email{panova@math.ucla.edu}
\date{ Received: July 15, 2011. Accepted: May 29, 2012}

\begin{document}
\maketitle
\begin{abstract}
We consider a new kind of straight and shifted plane partitions/Young tableaux --- ones whose diagrams are no longer of partition shape, but rather Young diagrams with boxes erased from their upper right ends. We find formulas for the number of standard tableaux in certain cases, namely a shifted staircase without the box in its upper right corner, i.e. truncated by a box, a rectangle truncated by a staircase  and a rectangle truncated by a square minus a box. The proofs involve finding the generating function of the corresponding plane partitions using interpretations and formulas for sums of restricted Schur functions and their specializations. The number of standard tableaux is then found as a certain limit of this function. 
\end{abstract}

\section{Introduction}

A central topic in the study of tableaux and plane partition are their enumerative properties. The number of standard Young tableaux (SYT) is given by the famous hook-length formula of Frame, Robinson and Thrall, a similar formula exists for shifted standard tableaux. However, most tableaux and posets in general do not have such nice enumerative properties, for example their main generalizations as skew tableaux are not counted by any product type formulas. 

In this paper we find product formulas for special cases of a new type of tableaux and plane partitions, ones whose diagrams are not straight or shifted Young diagrams of integer partitions. The diagrams in question are obtained by removing boxes from the north-east corners of a straight or shifted Young diagram and we say that the shape has been truncated by the shape of the boxes removed. We discover formulas for the number of tableaux of specific truncated shapes: shifted staircase truncated by one box in Theorem~\ref{thm1}, rectangle truncated by a staircase shape Theorem ~\ref{thm2} and  rectangle truncated by a square minus a box in Theorem~\ref{thm3}; these are illustrated as
\ytableausetup{boxsize=1ex}
$$\ydiagram{6,1+6,2+5,3+4,4+3,5+2,6+1}\;, \qquad \ydiagram{5,6,7,8,9,9,9} \; , \qquad \ydiagram{5,5,5,6,9,9,9}\; .$$
\noindent
The proofs rely on several steps of interpretations.   
 Plane partitions of truncated shapes are interpreted as (tuples of) SSYTs, which translates the problem into specializations of sums of restricted Schur functions. The number of standard tableaux is found as a polytope volume and then a certain limit of these specializations (generating function for the corresponding plane partitions). The computations involve among others complex integration and the Robinson-Schensted-Knuth correspondence.

The consideration of these objects started after R.~Adin and Y.~Roichman asked for a formula for the number of linear extensions of the poset of triangle-free triangulations, which are equivalent to standard tableaux of shifted straircase shape with upper right corner box removed, \cite{Roich}. We find and prove the formula in question, namely 
\begin{theorem}\label{thm1}
The number of shifted standard tableaux of shape $\delta_n \setminus \delta_1$ is equal to
$$g_n \frac{C_n C_{n-2}}{2 C_{2n-3}},$$
where $g_n=\frac{\binom{n+1}{2}!}{\prod_{0\leq i<j \leq n} (i+j)}$ is the number of shifted staircase tableaux of shape $(n,n-1,\ldots,1)$ and $C_m=\frac{1}{m+1}\binom{2m}{m}$ is the $m-$th Catalan number.
\end{theorem}
We also prove formulas for the cases of rectangles truncated by a staircase.
\begin{theorem}\label{thm2}
The number of standard tableaux of truncated straight shape $\smash{\underbrace{(n,n,\ldots,n)}_m}\setminus \delta_k$ (assume $n\leq m$), is
\begin{align}
\binom{mn-\binom{k+1}{2}}{m(n-k-1)} &f_{(n-k-1)^m} g_{(m,m-1,\ldots,m-k)}\frac{E_1(k+1,m,n-k-1)}{E_1(k+1,m,0)}, 
\end{align}
where $$ E_1(r,p,s)=\begin{cases} \prod_{r<l<2p-r+2}  \frac{1}{(l+2s)^{r/2}} \prod_{2\leq l\leq r}
\frac{1}{((l+2s)(2p-l+2+2s))^{\lfloor l/2 \rfloor}}, \; &r \text{ even}, \\  
\frac{((r-1)/2+s)!}{(p-(r-1)/2+s)!}E_1(r-1,p,s) , \; &r \text{ odd}.\end{cases}$$
\end{theorem}
\begin{theorem}\label{thm3}
The number of standard truncated tableaux of shape $n^m\setminus (k^{k-1},k-1)$ is
$$\binom{nm-k^2+1}{m(n-k)} \binom{k(m+n-2k)+1}{kn-k^2}^{-1}f_{(m^{n-k})}
f_{((m-k)^{k})}.$$
\end{theorem}
We will also exhibit connections with boxed plane partitions, as the generating function we use is the same as the volume generating function for boxed plane partitions.

We should note that numerical data suggests there are almost no other truncated shapes with product-type formulas. We see this because the number of standard tableaux of a shape of size $n$ has very large prime factors (e.g. on the order of $10^6$) even for small (below 50) values of $n$. For example, there does not seem to be a product type formula for the immediate generalizations of the shapes above, e.g. shifted truncated shape $\delta_n\setminus \delta_k$ for $k>1$, or rectangles truncated by squares minus a staircase $n^m \setminus (k^r,k-1,\ldots,r+1,r)$. The lack of hypothetical general formulas prevents us from applying the usual, generally induction based, techniques used for proving the hook-length formulas enumerating standard Young tableaux.  

It might be interesting to investigate what other properties of SYTs, besides product-form formulas, have analogues in tableaux of truncated shapes. The presence of such formulas might suggest connections to representation theory and/or combinatorial properties like RSK. While almost all other truncated shapes lack product-type formulas, it is conceivable that other closed form formulas, e.g. determinantal, could exist. 

 In \cite{Roich2} Adin, King and Roichman have independently found a formula for the case of shifted staircase truncated by one box and rectangle truncated by a square minus a box by methods different from the methods developed here. 

\section{Background and definitions}
\ytableausetup{boxsize=2ex}
We will review some basic facts and definitions from the classical theory of symmetric functions and Young tableaux which will be used later in this paper. The full description of these facts, as well as proofs and much more can be found in \cite{EC2} and \cite{Mac}.

An integer \textbf{partition} $\lambda$ of $n$ is a weakly decreasing sequence of nonnegative integers $\lambda=(\lambda_1,\lambda_2,\ldots,\lambda_l)$ which sum up to $n$, i.e. $\lambda_1+\cdots+\lambda_l=n$. The length of $\lambda$, $l(\lambda)$, denotes the number of parts of $\lambda$, in this case it is $l$. Partitions are depicted graphically as \textbf{Young diagrams}, a left justified array of $n=|\lambda|$boxes, such that row $i$ counted from the top has $\lambda_i$ boxes. For example, if $\lambda=(5,3,2)$ then $l(\lambda)=3$ and it's Young diagram is \ydiagram{5,3,2}. 
We will mean the Young diagram of $\lambda$ when we refer to $\lambda$ as a shape. A \textbf{skew} partition (and thus shape) $\lambda/\mu$ is the collection of boxes which belong to $\lambda$ but not $\mu$ when drawn with coinciding top left corners. E.g. the diagram of $(5,4,3,1)/(3,2)$ is $\ydiagram{3+2,2+2,3,1}$.

A \textbf{semi-standard Young tableaux (SSYT)} of shape $\lambda$ and type $\alpha=(\alpha_1,\alpha_2,\ldots)$, where $\alpha_1+\alpha_2+\cdots=|\lambda|$  is a filling of the boxes of the Young diagram of $\lambda$ with integers, such that the numbers weakly increase along each row and strictly decrease down each column, and there are exactly $\alpha_i$ numbers equal to $i$. A \textbf{standard Young tableaux (SYT)} is an SSYT of type $(1^n)=(\underbrace{1,\ldots,1}_n)$. A plane partition of shape $\lambda$ is a filling with nonnegative integers of the diagram of $\lambda$, such that the numbers decrease weakly along rows and down columns. The same definitions carry over verbatim when the shape is skew, i.e. $\lambda/\mu$.

\ytableausetup{boxsize=3ex}
\begin{tabular}{p{2in}p{2in}p{2in}}
\ytableaushort{112335,2344,456} &  \ytableaushort{123579,48{10}{12},6{11}{13}} & \ytableaushort{554331,5442,431} \\
SSYT of type $(2,2,3,3,2,1)$  & SYT of shape $(6,4,3)$& plane partition \\
\end{tabular}

The \textbf{Schur function} $s_{\lambda/\mu}(x_1,\ldots,x_l)$ is defined as
$$s_{\lambda/\mu}(x_1,\ldots,x_l) = \sum_{T: \sh(T)=\lambda/\mu} x^T,$$
where the sum is over all SSYTs $T$ of shape $\lambda/\mu$ and entries $1,2,\ldots,l$ and $x^T=\prod x_i^{\alpha_i}$ if the type of $T$ is $\alpha$. Another definition for Schur functions of shape $\lambda$ is Schur's determinantal formula:
$$s_{\lambda}(x_1,\ldots,x_k) = \frac{ \det[ x_i^{\lambda_j+l-j}]_{i,j=1}^l}{\det[ x_i^{l-j}]_{i,j=1}^l} = \frac{a_{\lambda+\delta_l}(x_1,\ldots,x_l)}{a_{\delta_{l-1}}(x_1,\ldots,x_l)},$$
where $a_{\beta}(x_1,\ldots,x_l) = \det[ x_i^{\beta_j}]_{i,j=1}^l$ and $\delta_l = (l,l-1,\ldots,1)$. 
The Schur functions belong to the ring of symmetric functions, i.e. functions invariant under permutations of the variables. Other symmetric functions we will refer to in this paper are the \textbf{homogenous symmetric functions} $h_{\lambda}(x_1,\ldots,x_l)$ defined as
$$h_p = \sum_{1\leq i_1 \leq i_2 \cdots \leq i_l \leq l} x_{i_1}x_{i_2}\cdots x_{i_l}, \qquad h_{\lambda}=\prod_{i=1}^{l(\lambda)} h_{\lambda_i}.$$

These functions satisfy a lot of remarkable identities, some of which we will mention and use later here. 
We have that 
\begin{equation}\label{sum_homog}\sum_{l \geq 0} h_l(x_1,\ldots,x_k)t^l = \prod_{i=1}^k \frac{1}{1-tx_i}.\end{equation}
The Schur functions satisfy Cauchy-type identities, namely
\begin{equation}\label{cauchy}
\sum_{\lambda}s_{\lambda}(x_1,\ldots)s_{\lambda}(y_1,\ldots)=\prod_{i,j}\frac{1}{1-x_iy_j}\quad \text{and} \quad \sum_{\lambda} s_{\lambda}(x_1,\ldots)=\prod_{i<j} \frac{1}{1-x_ix_j}\prod_i \frac{1}{1-x_i}.
\end{equation}
While these identities can proven algebraically, the more elegant proof is via a bijection between pairs of SSYTs of same shape $(P,Q)$, $\sh(P)=\sh(Q)$ and matrices $A=[a_{ij}]$ of nonnegative integer entries --- the \textbf{Robinson-Schensted-Knuth} correspondence. We refer the reader to \cite{EC2} for a full description of the bijection, proofs and properties. The most basic property is that if the row sums of $A$ are $\alpha$, i.e. $\sum_j a_{ij} =\alpha_i$, and the column sums are $\beta$, i.e. $\sum_i a_{ij}=\beta_j$, then the type of $P$ is $\beta$ and the type of $Q$ is $\alpha$. 
RSK translates many combinatorial properties. For example, if $A$ is a permutation matrix corresponding to the permutation $w$, and $\sh(P)=\sh(Q)=\lambda$, then the length of the longest increasing subsequence of $w$ is equal to $\lambda_1$ and the length of the longest decreasing subsequence is equal to $l(\lambda)$. This fact extends to general matrices (\textbf{Schensted's theorem}), where we replace the permutation by a double array listing $\begin{pmatrix} i \\j\end{pmatrix}$ each $a_{ij}$ times in order of increasing $i$s and then $j$s and `increasing subsequence' is replaced by `weakly increasing subsequence of  $\begin{pmatrix} i \\j\end{pmatrix}$s' and `decreasing subsequence' is replaced by `strictly decreasing' in both $i$ and $j$. 

We will illustrated the statements above by an example. Let $A= \begin{pmatrix} 1&0&2\\ 0 & 2 & 0\\ 1 & 1 & 0\end{pmatrix}$, then the corresponding double array is $w_A=\begin{pmatrix} 1&1&1 &2&2&3& 3 \\ 1 & 3 & 3 & 2& 2&1 & 2 \end{pmatrix}$. Applying RSK to $A$ we get $(P,Q)$ with 
$$P = \ytableaushort{1122,23,3} \quad \text{and}\quad Q=\ytableaushort{1113,22,3}.$$ We have that $\lambda=\sh(P)=\sh(Q)=(4,2,1)$ and a longest increasing subsequence of $w_A$ is indeed of length $4=\lambda_1$:  $\begin{pmatrix} \bu{1}&1&1 &\bu{2}&\bu{2}&3& \bu{3} \\ \bu{1} & 3 & 3 & \bu{2}& \bu{2}&1 & \bu{2} \end{pmatrix}$ and the longest decreasing subsequence is of length $l(\lambda)=3$: 
$\begin{pmatrix} 1&1&\bu{1} &\bu{2}&2&\bu{3}& 3 \\ 1 & 3 & \bu{3} &\bu{ 2}& 2&\bu{1} & 2 \end{pmatrix}$.

\ytableausetup{boxsize=2ex}
An important and remarkable combinatorial property of standard Young tableaux is the existence of a product type formula for the number of SYTs of a given shape $\lambda$, denoted $f_{\lambda}$. The so called \textbf{hook-length formulas} for the number of straight  standard Young tableaux (SYT) of shape $\lambda$ are:
$$ f_{\lambda}=\frac{|\lambda|!}{\prod_{u \in \lambda} h_u},\qquad \text{ hook $h_u$: }\ytableaushort{{}{}{}{}{}{},{}{*(gray)u}{*(gray)}{*(gray)}{*(gray)},{}{*(gray)}{}{},{}{*(gray)},{}}$$
and the number of standard Young tableaux of shifted shape $\lambda$ 
$$g_{\lambda}=\frac{|\lambda|!}{\prod_u h_u}\qquad  \text{ hook $h_u$: }\ytableaushort{{}{}{*(gray)u}{*(gray)}{*(gray)}{*(gray)}{*(gray)},\none{}{*(gray)}{}{}{}{},\none\none {*(gray)}{}{}{},\none\none\none {*(gray)}{*(gray)}{*(gray)},\none\none\none\none{}}$$

We are now going to define our main objects of study. 

Let $\lambda =(\lambda_1,\lambda_2,\ldots)$ and $\mu=(\mu_1,\mu_2,\ldots)$ be integer partitions, such that $\lambda_i \geq \mu_i$.
A \textbf{straight diagram of truncated shape $\lambda \setminus \mu$} is a left justified array of boxes, such that row $i$ has $\lambda_i - \mu_i$ boxes. If $\lambda$ has no equal parts we can define a \textbf{shifted diagram of truncated shape $\lambda \setminus \mu$} as an array of boxes, where row $i$ starts one box to the right of the previous row $i-1$ and has $\lambda_i-\mu_i$ number of boxes. 
For example
\ytableausetup{boxsize=1.5ex}
$$D_1=\ydiagram{3,4,6,6,5}\; ,\qquad D_2 = \ydiagram{3,1+5,2+6,3+2}$$
$D_1$ is of straight truncated shape $(6,6,6,6,5) \setminus (3,2)$ and $D_2$ is of shifted truncated shape $(8,7,6,2)\setminus(5,2)$.

We define standard and semistandard Young tableaux and plane partitions of truncated shape the usual way except this time they are fillings of truncated diagrams. A \textbf{standard truncated Young tableaux} of shape $\lambda \setminus \mu$ is a filling of the corresponding truncated diagram with the integers from 1 to $|\lambda|-|\mu|$, such that the entries across rows and down columns are increasing and each number appears exactly once. A \textbf{plane partition of truncated shape} $\lambda \setminus \mu$ is a filling of the corresponding truncated diagram with integers such that they weakly increase along rows and down columns. For example
\ytableausetup{boxsize=2ex}
$$T_1=\ytableaushort{124,\none 3578,\none\none 69}\; , \quad T_2=\ytableaushort{113,\none 2445,\none\none 56}\; ,\quad T_3=\ytableaushort{113,\none 1446,\none\none 45}$$ 
are respectively a standard Young tableaux (SYT), a semi-standard Young tableaux (SSYT) and a plane partition (PP) of shifted truncated shape $(5,4,2)\setminus(2)$. We also define \textbf{reverse} PP, SSYT, SYT by reversing all inequalities in the respective definitions (replacing weakly/strictly increasing with weakly/strictly decreasing). For this paper it will be more convenient to think in terms of the reverse versions of these objects.

In this paper we will consider truncation by staircase shape $\delta_k=(k,k-1,k-2,\ldots,1)$ of shifted staircases and straight rectangles. We will denote by $T[i,j]$ the entry in the box with coordinate $(i,j)$ in the diagram of $T$, where $i$ denotes the row number counting from the top and $j$ denotes the $j-$th box in this row counting from the beginning of the row.

For any shape $D$, let
$$F_D(q) = \sum_{T\; : \;PP,\;\sh(T)=D} q^{\sum T[i,j]}$$
be \textbf{the generating function} for the sum of the entries of all plane partitions of shape $D$.

\section{A bijection with skew SSYT}
\ytableausetup{boxsize=2.5ex}
We will consider a map between truncated plane partitions and skew Semi-Standard Young Tableaux which will enable us to enumerate them using Schur functions. 

As a basic setup for this map we first consider truncated shifted plane partitions of staircase shape $\delta_n \setminus \delta_k$. Let $T$ be such a plane partition. Let $\lambda^j = (T[1,j], T[2,j],\ldots, T[n-j,j])$ - the sequence of numbers in the $j$th diagonal of $T$. For example, if $T=\ytableaushort{865,\none 643,\none\none 42,\none\none\none 1}\;$, then $\lambda^1=(8,6,4,1)$, $\lambda^2=(6,4,2)$ and $\lambda^3=(5,3)$. 

Let $P$ be a reverse skew tableaux of shape $\lambda^1 / \lambda^{n-k}$, such that the entries filling the subshape $\lambda^j / \lambda^{j+1}$ are equal to $j$, i.e.\ it corresponds to the sequence $\lambda^{n-k} \subset \lambda^{n-k-1}\subset \cdots \subset \lambda^1$. The fact that this is all well defined follows from the inequalities that the $T[i,j]$'s satisfy by virtue of $T$ being a plane partition. Namely, $\lambda^{j+1} \subset \lambda^j$, because $\lambda^j_i = T[i,j] \geq T[i,j+1] = \lambda^{j+1}_i$. Clearly the rows of $P$ are weakly decreasing. The columns are strictly decreasing, because for each $j$ the entries $j$ in the $i$th row of $P$ are in positions $T[i,j+1]+1$ to $T[i,j]< T[i-1, j+1]+1$, so they appear strictly to the left of the $j$s in the row above ($i-1$).

Define $\phi(T)=P$, $\phi$ is the map in question. Given a reverse skew tableaux $P$ of shape $\lambda \setminus \mu$ and entries smaller than $n$ we can obtain the inverse shifted truncated plane partition $T= \phi^{-1}(P)$ as
$T[i,j] = \max (s|P[i,s]\geq j)$; if no such entry of $P$ exists let $s=0$.

For example we have that
\begin{equation}
T=\ytableaushort{8765,\none 7543,\none\none 532,\none\none\none 31,\none\none\none\none 1}\qquad \xrightarrow{\phi} \qquad
 P = \ytableaushort{\none\none\none\none\none 321,\none\none\none 3211,33211,211,1}
\end{equation}

Notice that 
\begin{equation} \label{truncatedtoschur}
\sum T[i,j] = \sum P[i,j] + |\lambda^{n-k}|(n-k),
\end{equation}
where $\sh(P) = \lambda^1 /\lambda^{n-k}$. 

The map $\phi$ can be extended to any truncated shape, then the image will be tuples of SSYTs with certain restrictions. For the purposes of this paper we will extend it to truncated plane partitions of shape $(n^m) \setminus \delta_k$ as follows.   

Let $T$ be a plane partition of shape $n^m \setminus \delta_k$ and assume that $n\leq m$ (otherwise we can reflect about the main diagonal). Let $\lambda=(T[1,1],T[2,2],\ldots,T[n,n])$, $\mu =(T[1,n-k],T[2,n-k+1],\ldots,T[k+1,n])$ and let $T_1$ be the portion of $T$ above and including the main diagonal, hence of shifted truncated shape $\delta_n \setminus \delta_k$, and $T_2$ the transpose of the lower portion including the main diagonal, a shifted PP of shape $(m,m-1,\ldots,m-n+1)$.

Extend $\phi$ to $T$ as $\phi(T) = (\phi(T_1),\phi(T_2))$. Here $\phi(T_2)$ is a SSYT of at most $n$ rows (shape $\lambda$) and filled with $[1,\ldots,m]$ the same way as in the truncated case. 

As an example with $n=5,m=6,k=2$ we have
\begin{gather*}
\begin{aligned}
T=\ytableaushort{764,6644,44332,43222,32211,21111} && T_1=\ytableaushort{764,\none 644,\none\none 332,\none\none\none 22,\none\none\none\none 1}&& T_2=\ytableaushort{764432,\none 64321,\none\none 3221,\none\none\none 211,\none\none\none\none 11},
\end{aligned} \\
\begin{aligned}
\phi(T) =\quad & \left( \ytableaushort{\none\none\none\none 221,\none\none\none\none 11,\none\none 2,22,1},\ytableaushort{6654221,543211,431,31,2} \right) 
\end{aligned}
\end{gather*}

We thus have the following
\begin{prop}\label{sumofentries}
The map $\phi$ is a bijection between shifted truncated plane partitions $T$ of shape $\delta_n \setminus \delta_k$ filled with nonnegative integers and (reverse) skew semi-standard Young tableaux with entries in $[1,\ldots,n-k-1]$ of shape $\lambda /\mu$ with $l(\lambda)\leq n$ and $l(\mu) \leq k+1$. Moreover, 
$\sum_{i,j} T[i,j] = \sum_{i,j} P[i,j] + |\mu|(n-k)$. 
Similarly $\phi$ is also a bijection between truncated plane partitions $T$ of shape $n^m \setminus \delta_k$ and pairs of SSYTs $(P,Q)$, such that $\sh(P)=\lambda / \mu$, $\sh(Q)=\lambda$ with $l(\lambda) \leq n$, $l(\mu)\leq k+1$ and $P$ is filled with $[1,\ldots,n-k-1]$, $Q$ with $[1,\ldots,m]$. Moreover, $\sum T[i,j] = \sum P[i,j] + \sum Q[i,j] -|\lambda|+|\mu|(n-k)$. 
\end{prop} 

\section{Schur function identities}

We will now consider the relevant symmetric function interpretation arising from the map $\phi$. Substitute the entries $1,\ldots$ in the skew SSYTs in the image with respective variables $x_1,\ldots$ and $z_1,\ldots$. The idea is to evaluate the resulting expressions at certain finite specializations for $x$ and $z$( e.g. $x=(q,q^2,\ldots,q^{n-k-1})$ and $z=(1,q,q^2,\ldots,q^{m-1})$) to obtain generating functions for the sum of entries in the truncated plane partitions which will later allow us to derive enumerative results. 

For the case of shifted truncated shape $\delta_n \setminus \delta_k$ we have the corresponding sum
\begin{align}\label{single}
S_{n,k}(x;t)=\sum_{\lambda, \mu|l(\lambda) \leq n, l(\mu)\leq k+1}s_{\lambda/\mu}(x_1,\ldots,x_{n-k-1})t^{|\mu|},
\end{align}
and for the straight truncated shape $n^m \setminus \delta_k$
\begin{align}\label{double}
D_{n,m,k}(x;z;t)=\sum_{\lambda,\mu| l(\lambda)\leq n, l(\mu)\leq k+1} s_{\lambda}(z)s_{\lambda/\mu}(x)t^{|\mu|}.
\end{align}

We need to find formulas when $x_i=0$ for $i>n-k-1$ and $z_i=0$ for $i>m$. Keeping the restriction $l(\mu)\leq k+1$ we have that $s_{\lambda/\mu}(x)=0$ if $l(\lambda)>n$ and this allows us to drop the length restriction on $\lambda$ in both sums. 

From now on the different sums will be treated separately.
Consider another set of variables $y=(y_1,\ldots,y_{k+1})$ which together with $(x_1,\ldots,x_{n-k-1})$ form a set of $n$ variables.
Using Cauchy's first identity in \eqref{cauchy} we have that 
\begin{align*}
&\sum_{\lambda|l(\lambda)\leq n} s_{\lambda}(x_1,\ldots,x_{n-k-1},ty_1,\ldots,ty_{k+1}) \\
&= \prod \frac{1}{1-x_i} \prod_{i<j\leq n-k-1} \frac{1}{1 - x_i x_j} \prod_{i<j\leq k+1} \frac{1}{1-t^2y_iy_j} \prod_{i,j} \frac{1}{1 - x_ity_j} \prod \frac{1}{1-ty_i},
\end{align*}
where the length restriction drops because $s_{\lambda}(u_1,\ldots,u_n)=0$ when $l(\lambda)>n$. 
On the other hand we have that for any $\lambda$
$ s_{\lambda}(x_1,\ldots,x_{n-k-1},ty_1,\ldots,y_{k+1}t)
= \sum_{\mu|l(\mu) \leq k+1} s_{\lambda/\mu}(x_1,\ldots, x_{n-k-1}) s_{\mu}(ty_1,\ldots,ty_{k+1}),$
since again $s_{\mu}(yt)=0$ if $l(\mu) >k+1$. 
We thus get
\begin{align}\label{step1}
\prod \frac{1}{1-x_i} \prod_{i<j\leq n-k-1} \frac{1}{1 - x_i x_j} & \prod_{i<j\leq k+1} \frac{1}{1-t^2y_iy_j} \prod_{i,j} \frac{1}{1 - x_ity_j} \prod \frac{1}{1-ty_i} \\
&=\sum_{\lambda,\mu|l(\mu) \leq k+1} s_{\lambda/\mu}(x_1,\ldots, x_{n-k-1}) s_{\mu}(y_1,\ldots,y_{k+1})t^{|\mu|}\notag
\end{align}
We now need to extract the coefficients of $s_{\mu}(y)$ from both sides of \eqref{step1} to obtain a formula for $S_{n,k}(x;t)$. To do so we will use the determinantal formula for the Schur functions, namely that
\begin{align*}
s_{\nu}(u_1,\ldots,u_p)=
\frac{a_{\nu+\delta_p}(u)}{a_{\delta_p}(u)}=
\frac{\det[ u_i^{\nu_j+p-j}]_{i,j=1}^{p}}{\det[u_i^{p-j}]_{i,j=1}^{p}}.\\
\end{align*}
We also have that $a_{\delta_p}(z_1,\ldots,z_p) = \prod_{i<j}(z_i-z_j)$. Substituting $s_{\mu}(y)$ with $a_{\delta_{k+1}+\mu}(y)/a_{\delta_{k+1}}(y)$ in the right-hand side of \eqref{step1} and multiplying both sides by $a_{\delta_{k+1}}(y)$ we obtain
\begin{align} \label{step2}
\sum_{\lambda,\mu} s_{\lambda/\mu}(x_1,\ldots, x_{n-k-1})& a_{\mu+\delta_{k+1}}(y_1,\ldots,y_{k+1})t^{|\mu|} \\
&= \prod \frac{1}{1-x_i} \prod_{i<j\leq n-k-1} \frac{1}{1 - x_i x_j} \prod_{i<j\leq k+1} \frac{y_i-y_j}{1-t^2y_iy_j} \prod_{i,j} \frac{1}{1 - x_ity_j} \prod \frac{1}{1-ty_i}  \notag
\end{align}  
Observe that $a_{\alpha}(u_1,\ldots,u_p) = \sum_{w \in S_p} \sgn(w) u_{w_1}^{\alpha_1}\cdots u_{w_p}^{\alpha_p}$ with $\alpha_i>\alpha_{i+1}$ has exactly $p!$ different monomials in $y$, each with a different order of the degrees of $u_i$ (determined by $w$). Moreover, if $\alpha \neq \beta$ are partitions of distinct parts, then $a_{\alpha}(u)$ and $a_{\beta}(u)$ have no monomial in common. 
Let $$A(u) = \sum_{\alpha| \alpha_i>\alpha_{i+1}} \sum_{w\in S_p} \sgn(w) u_{w_1}^{\alpha_1}\cdots u_{w_p}^{\alpha_p}.$$ 
For every $\beta$ of strictly decreasing parts, every monomial in $a_{\beta}(u)$ appears exactly once and with the same sign in $A(u)$, so $a_{\beta}(u)A(u^{-1})$ has coefficient at $u^0$ equal to $p!$.
Therefore we have
\begin{align}
[y^0]&\left( \sum_{\lambda,\mu} s_{\lambda/\mu}(x_1,\ldots, x_{n-k-1}) a_{\mu+\delta_{k+1}}(y_1,\ldots,y_{k+1})t^{|\mu|} A(y^{-1}) \right) \\
&= (k+1)!\sum_{\lambda, \mu|l(\mu)\leq k+1} s_{\lambda/\mu}(x_1,\ldots, x_{n-k-1})t^{|\mu|} = (k+1)!S_{n,k}(x;t).
\end{align}

Using the fact that $a_{\alpha}(u) =s_{\nu}(u)a_{\delta}(u)$ for $\nu=\alpha-\delta_p$ and Cauchy's formula for the sum of the Schur functions, second equation in \eqref{cauchy}, we have
$$A(u) = \sum_{\nu} s_{\nu}(u) a_{\delta_p}(u) = \prod \frac{1}{1-u_i} \prod_{i<j} \frac{(u_i-u_j)}{1-u_iu_j}.$$

In order for $A(y^{-1})$ and $\sum s_{\mu}(y) t^{|\mu|}$ to converge we need $1< |y_i| < |t^{-1}|$ for every $i$. For $S_{n,k}(x;t)$ to also converge, let $|x_j|<1$ for all $j$. 

We also have that for any doubly infinite series $f(y)$, $[y^0]f(y) = \frac{1}{2\pi i}\int_C f(y)y^{-1} dy$, on a circle $C$ in the $\mathbb{C}$ plane centered at 0 and within the region of convergence of $f$. Rewriting $A(y^{-1})$ in terms of $y$   similar to \eqref{step2} (where a sign of $(-1)^{\binom{k+1}{2}}$ appears) we obtain the formula for $S_{n,k}(x;t)$ through a complex integral. 
\begin{prop}\label{integral} We have that
\begin{align*}
S_{n,k}(x;t) =&\frac{(-1)^{\binom{k+1}{2}}}{(k+1)!}  \prod \frac{1}{1-x_i} \prod_{i<j\leq n-k-1} \frac{1}{1 - x_i x_j} \\
 \frac{1}{(2\pi i)^{k+1}}&  \int_{T} \prod_{i<j\leq k+1} \frac{(y_i-y_j)^2}{1-t^2y_iy_j} \prod_{i,j} \frac{1}{1 - x_ity_j} \prod \frac{1}{1-ty_i} \prod \frac{1}{y_i-1} \prod_{i<j} \frac{1}{y_iy_j-1} dy_1\cdots dy_{k+1}, 
\end{align*}
where $T=C_1 \times C_2 \times \cdots C_p$ and $C_i =\{z\in\mathbb{C}||z|=1+\epsilon_i\}$ for $\epsilon_i <|t^{-1}|-1$.
\end{prop}

For the case of \textbf{straight shapes} we are looking for formulas for 
$$\sum_{\lambda,\mu} s_{\lambda}(z)s_{\lambda/\mu}(x)t^{|\mu|}$$
with certain length restrictions for $\lambda$ and $\mu$ and certain values of $x,z,t$.
Without the length restrictions we have by Cauchy's identity, \eqref{cauchy}, that
$$\sum_{\lambda,\mu}s_{\lambda}(z) s_{\lambda/\mu}(x)s_{\mu}(y) = \sum_{\lambda} s_{\lambda}(z)s_{\lambda}(x,y) =\prod \frac{1}{1-z_ix_j} \prod \frac{1}{1-z_iy_j}.$$
Expanding the second product on the right-hand side using Cauchy's identity again we obtain
\begin{align*}
\sum_{\lambda,\mu} s_{\lambda}(z)s_{\lambda/\mu}(x)s_{\mu}(y)= \prod \frac{1}{1-z_ix_j}\sum_{\mu}s_{\mu}(z)s_{\mu}(y)
\end{align*}
Comparing the coefficients for $s_{\mu}(y)$ on both sides we get for any $\mu$
\begin{align}\label{fixedmu}
\sum_{\lambda} s_{\lambda}(z)s_{\lambda/\mu}(x)=\prod \frac{1}{1-z_ix_j}\sum_{\mu}s_{\mu}(z).
\end{align}
For the \textbf{straight shape case} we thus find
$$D_{n,m,k}(x;z;t)=\sum_{\lambda,\mu| l(\lambda)\leq n, l(\mu)\leq k+1} s_{\lambda}(z)s_{\lambda/\mu}(x)t^{|\mu|}=\prod \frac{1}{1-z_ix_j} \sum_{\mu,l(\mu)\leq k+1}s_{\mu}(z)t^{|\mu|},$$
where the length restriction on $\lambda$ becomes redundant when evaluating the sum at
$x
=(x_1,\ldots,x_{n-k-1})$ since then $s_{\lambda/\mu}(x)=0$ when $l(\lambda)>l(\mu)+n-k-1\leq k+1+n-k-1=n$.

\begin{prop}\label{d-formula}
We have that
\begin{align*}
D_{n,m,k}(x,z;t) = \prod_{i,j} \frac{1}{1-x_iz_j} \left(\sum_{\nu|l(\nu) \leq k+1} s_{\nu}(zt)\right).
\end{align*}
\end{prop}
For the purpose of enumeration of SYTs we will use this formula as it is. Even though there are formulas, e.g. of Gessel and King, for the sum of Schur functions of restricted length in the form of determinants or infinite sums, they would not give the enumerative answer any more easily.

\section{A polytope volume as a limit}

Plane partitions of specific shape (truncated or not) of size $N$ can be viewed as integer points in a cone in $\mathbb{R}^N$. Let $D$ be the diagram of a plane partition $T$, the coordinates of $\mathbb{R}^{|D|}$ are indexed by the boxes present in $T$. Then 
$$C_D=\{(\cdots,x_{i,j},\cdots)\in \mathbb{R}_{\geq 0}^N:\; [i,j]\in D, \; x_{i,j}\leq x_{i,j+1}\text{ if }[i,j+1]\in D, x_{i,j}\leq x_{i+1,j} \text{ if } [i+1,j]\in D\}$$ is the corresponding cone.

Let $P(C)$ be the section of a cone $C$ in $\mathbb{R}_{\geq 0}^N$ with the halfplane $H=\{x|\sum_{[i,j]\in D} x_{i,j}\leq 1\}$. Consider the standard tableaux of shape $T$; these correspond to all linear orderings of the coordinates in $C_D$ and thus also $P_D=P(C_D)$. Considering $T$ as a bijection $D\rightarrow [1,\ldots,N]$, $P_D$ is thus subdivided into chambers $\{x:0\leq x_{T^{-1}(1)}\leq x_{T^{-1}(2)}\leq \cdots x_{T^{-1}(N)}\}\cap H$ of equal volumes, namely $\frac{1}{N!}Vol(\Delta_N)$, where $\Delta_N = H \cap \mathbb{R}_{\geq 0}^N$ is the $N$-simplex. Hence the volume of $P_D$ is \begin{align}\label{volsyt}
Vol_{N-1}(P_D) = \frac{\# T: SYT,\sh(T)=D}{N!} Vol(\Delta_N).
\end{align}

The following lemma helps determine the volume and thus the number of SYTs of shape $D$.
\begin{lemma}
Let $P$ be a $d$-dimensional rational polytope in $\mathbb{R}_{\geq 0}^d$, such that its points satisfy $a_1+\cdots+a_d\leq 1$ for $(a_1,\ldots,a_d)\in P$, and let 
$$F_P(q) =\sum_{n} \sum_{(a_1,\ldots,a_d) \in nP\cap \mathbb{Z}^d } q^{n}.$$ 
We have that the $d$-dimensional volume of $P$ is $$Vol_d(P) = (\lim_{q \rightarrow 1} (1-q)^{d+1}F_P(q))\frac{1}{d!}.$$ 
\end{lemma}
\begin{proof}
Let $f_P(n) = \# \{(a_1,\ldots,a_d) \in nP\}$. Then
$F_P(q)= \sum_n f_P(n)q^n.$
Moreover, $Vol_d(P) = \lim_{n \rightarrow \infty} \frac{f_P(n)}{n^{d}}$ since subdividing an embedding of $P$ in $\mathbb{R}^{d}$  into $(d)$-hypercubes with side $\frac{1}{n}$  we get $f_P(n)$ cubes of total volume $f_P(n)/n^{d}$ which approximate $P$ as $n \rightarrow \infty$.

Moreover since $\lim_{n \rightarrow \infty} \frac{n^{d}}{(n+1)\cdots(n+d)}=1$ we also have $Vol_{d}(P) = \lim_{n \rightarrow \infty} \frac{f_P(n)}{(n+1)\cdots(n+d)}$.
Let $G(q) = \sum_n \frac{f_P(n)}{(n+1)\cdots(n+d)} q^{n+d}$. Then 
$$\tilde{G}(q)=(1-q)G(q) = \sum_n \underbrace{(\frac{f_P(n)}{(n+1)\cdots(n+d)} - \frac{f_P(n-1)}{(n)\cdots(n+d-1)})}_{b_n}q^{n+d}$$
and $\tilde{G}(1) = \sum b_n =\lim_{n\rightarrow \infty} b_1+\cdots +b_n= \lim_{n\rightarrow \infty}\frac{f_P(n)}{(n+1)\cdots(n+d)} = Vol_d(P)$. On the other hand 
\begin{align*}
Vol_d(P)&=\tilde{G}(1) = 
\lim_{q \rightarrow 1}(1-q)G(q) = 
\lim_{q \rightarrow 1} \frac{G'(q)}{(\frac{1}{1-q})'}\\
&=\cdots =\lim_{q \rightarrow 1} \frac{ G^{(d)}(q)}{ (\frac{1}{1-q})^{(d)} }= \lim_{q \rightarrow 1}\frac{ F_p(q)(1-q)^{d+1}}{d!}
\end{align*}
by L'Hopital's rule.
\end{proof}

Notice that if $P=P(C_D)$ for some shape $D$, then since $(a_1,\ldots,a_d) \in mP$ when $(a_1,\ldots,a_d) \in C_D$ and $a_1+\cdots+a_d \leq m$, we have that
\begin{align*}
F_P(q) =& \sum_{a \in C_D} q^{a_1+\cdots+a_d}(1+q+\cdots)
=\frac{1}{1-q}\sum_{a \in C_D\cap \mathbb{Z}^N} q^{|a|}=\frac{1}{1-q}\sum_{T:PP, \sh(T)=D} q^{\sum T[i,j]} =\frac{1}{1-q}F_D(q).
\end{align*}
Using \eqref{volsyt}, the fact that $Vol(\Delta_d) =\frac{1}{d!}$ and this Lemma we get the key fact to enumerating SYTs of truncated shapes using evaluations of symmetric functions.

\begin{prop}\label{limformula}
The number of standard tableaux of shape $D$ with $N$ squares is equal to
$$N!\lim_{q \rightarrow 1} (1-q)^{N}F_D(q).$$
\end{prop}

\section{Standard tableaux of shape shifted staircase truncated by a box }
We are now going to use Propositions \ref{integral} and \ref{limformula} to find the number of standard shifted tableaux of truncated shape $\delta_n \setminus \delta_1$. Numerical results show that a product formula for the general case of truncation by $\delta_k$ does not exist.

First we will evaluate the integral in Proposition \ref{integral} by iteration of the Residue theorem.
For simplicity, let $u_0=t$ $u_i=tx_i$, so the integral becomes 
$$\frac{1}{(2\pi i)^2}\int_{T} 
  \frac{(y_1-y_2)^2}{1-t^2y_1y_2}  \frac{1}{y_2-1} \frac{1}{y_1-1} \frac{1}{y_1y_2-1}\prod_{i\geq 0,j=1,2} \frac{1}{1 - u_iy_j}  dy_1 dy_2$$
Integrating by $y_1$ we have poles at $1$,$y_2^{-1}t^{-2}$, $y_2^{-1}$ and $u_i^{-1}$. Only $1$ and $y_2^{-1}$ are inside $C_1$ and the respective residues are
\begin{align*}
Res_{y_1=1} 
  = -\prod_{i\geq 0} \frac{1}{1-u_i} \frac{1}{2\pi i}\int_{C_2} \frac{1}{(1-t^2y_2)}\prod \frac{1}{1-u_iy_2} dy_2 = 0,
  \end{align*}
 since now the poles for $y_2$ are all outside $C_2$.
 
 For the other residue we have that
 \begin{align*}
 Res_{y_1=y_2^{-1}} &= \frac{1}{2\pi i}\int_{C_2} 
  \frac{(y_2^{-1}-y_2)^2}{1-t^2}  \frac{1}{y_2-1} \frac{1}{y_2^{-1}-1} \frac{1}{y_2} \prod_{i\geq 0} \frac{1}{1 - u_iy_2}\prod_{i\geq 0} \frac{1}{1 - u_iy_2^{-1}}   dy_2 \\
 &= \frac{1}{1-t^2}\frac{1}{2\pi i} \int_{C_2} (1+y_2)^2y_2^{n-3}\prod_{i\geq 0} \frac{1}{y_2 - u_i} \prod_{i\geq 0} \frac{1}{1 - u_iy_2} dy_2
 \end{align*}
 
 If $n\geq 3$ the poles inside $C_2$ are exactly $y_2=u_i$ for all $i$ and so we get a final answer
 \begin{align*}
 -\sum_{i\geq 0} \frac{1}{1-t^2}(1+u_i)^2u_i^{n-3}\prod_{j\neq i} \frac{1}{u_i - u_j} \prod_{j\geq 0} \frac{1}{1 - u_iu_j}
 \end{align*}
 and
 $$S_{n,1}(x;t) =
 \frac{(-1)^{\binom{k+1}{2}+1}}{(k+1)!}
   \prod \frac{1}{1-x_i} 
   \prod_{1\leq i<j\leq n-k-1} \frac{1}{1 - x_i x_j} 
   \sum_{i\geq 0} \frac{1}{1-t^2}(1+u_i)^2u_i^{n-3}
   \prod_{j\neq i} \frac{1}{u_i - u_j}
    \prod_{j\geq 0} \frac{1}{1 - u_iu_j},$$
 where $u_i =tx_i$ with $x_0=1$.
 
We can simplify the sum above as follows.
Notice that for any $p$ variables $v=(v_1,\ldots,v_p)$ we have
$$\sum_{i=1}^{p} \frac{v_i^r}{\prod (v_i - v_j)} = \frac{ \sum_i (-1)^{i-1} v_i \prod_{s<l;l,s \neq i}(v_s-v_l)}{\prod_{s<l} (v_s-v_l)} = \frac{a_{(r-p+1)+\delta_p}(v)}{a_{\delta_p}(v)}=s_{(r-p+1)}(v)=h_{r-p+1}(v),$$
where $h_s(v)=\sum_{i_1\leq i_2\leq \cdots \leq i_s}v_{i_1}v_{i_2}\cdots v_{i_s}$ is the $s$-th homogenous symmetric function.

Then we have that 
\begin{align}
\sum_{i=0}^{n-k-1} \frac{u_i^s}{\prod (u_i - u_j)} \prod \frac{1}{1-u_iu_j} &=  \sum_{i=0}^{n-k-1} \frac{u_i^s}{\prod (u_i - u_j)}\sum_{p} u_i^ph_p(u)\notag \\
&=\sum_{p\geq 0}\sum_{i=0}^{n-k-1} \frac{u_i^{s+p}}{\prod (u_i - u_j)}h_p(u) \notag \\
&= \sum_{p \geq 0} h_{s-n+k+1+p}(u)h_p(u) = c_{s-n+k+1}(u),
\end{align}
where $c_i = \sum_{n\geq 0}h_nh_{n+i}$. 
We then have the new formulas
\begin{align}\label{s_simplified}
S_{n,1}(x;t) =
   \prod \frac{1}{1-x_i} 
   \prod_{1\leq i<j\leq n-k-1} \frac{1}{1 - x_i x_j} \frac{1}{1-t^2}(c_{1}(u)+c_{0}(u)),
\end{align}
where $(1+u_i)^2u_i^{n-3}=u_i^{n-3}+2u_i^{n-2}+u_i^{n-1}$ contributed to $c_{-1}(u)+2c_0(u)+c_1(u)$ and by its definition $c_{-1}=c_1$.
We are now ready to prove the following.
\begin{theorem}
The number of shifted standard tableaux of shape $\delta_n \setminus \delta_1$ is equal to
$$g_n \frac{C_n C_{n-2}}{2 C_{2n-3}},$$
where $g_n=\frac{\binom{n+1}{2}!}{\prod_{0\leq i<j \leq n} (i+j)}$ is the number of shifted staircase tableaux of shape $(n,n-1,\ldots,1)$ and $C_m=\frac{1}{m+1}\binom{2m}{m}$ is the $m-$th Catalan number.
\end{theorem}
\begin{proof}
We will use Proposition \ref{limformula} and the formula \eqref{s_simplified}.

By the properties of $\phi$ we have that for the shape $D=\delta_n \setminus \delta_k$,
$$F_D(q)=\sum_{T|\sh(T)=D} q^{\sum T[i,j]} = \sum_{P=\phi(T)} q^{(n-k)|\mu|+\sum P[i,j]}=S_{n,k}(q,q^2,\ldots,q^{n-k-1};q^{n-k}).$$

Now let $k=1$. For the formula in Proposition \ref{limformula} we have $N=\binom{n+1}{2}-1$ and plugging $x=(q,\ldots,q^{n-2})$, $t=q^{n-1}$ in \eqref{s_simplified} we get
$$\lim_{q\rightarrow 1}(1-q)^N F_D(q) = 
   \prod \frac{1}{1-q^i} 
   \prod_{1\leq i<j\leq n-2} \frac{1}{1 - q^{i+j}} \frac{1}{1-q^{2(n-1)}}(c_{1}(u)+c_{0}(u)),$$
where $u=(q^{n-1},q^n,\ldots,q^{2n-3}).$   

We need to determine $\lim_{q\rightarrow 1}(1-q)^{2n-3} c_s(u)$.
Let $c_s(x;y)=\sum_l h_l(x)h_{l+s}(y)$ where $x=(x_1,\ldots,x_n)$ and $y=(y_1,\ldots,y_m)$. We have that
\begin{align*}
c_s(x;y) &= \sum_l \sum_{i_1\leq\cdots \leq i_l; j_1\leq \cdots \leq j_{l+s}} x_{i_1}\cdots x_{i_l} y_{j_1}\cdots y_{j_{s+l}}\\
&= \sum_{p} h_s(y_1,\ldots,y_p)\sum_{P:(1,p)\rightarrow (n,m)} (-1)^{m+n-p - \#P}\sum_l h_l((xy)_P),
\end{align*} 
where the sum runs over all fully ordered collections of lattice points $P$ in between $(1,p)$ to $(n,m)$ and $(xy)_P = (x_{i_1}y_{j_1},\ldots)$ for $(i_1,j_1),\ldots \in P$ and the $(-1)$s indicate the underlying inclusion-exclusion process.
We also have that, by \eqref{sum_homog}, 
$$\sum_l h_l((xy)_P) = \frac{1}{\prod_{(i,j)\in P} (1-x_iy_j)}.$$
The degree of $1-q$ dividing the denominators after substituting $(x_i,y_j)=(u_i,u_j)=(q^{n-2+i},q^{n-2+j})$ for the evaluation of $c_s(u)$ is equal to the number of points in $P$. $\#P$ is maximal when the lattice path is from $(1,1)$ to $(n-1,n-1)$ and is saturated, so $\max(\#P)=2(n-1)-1=2n-3$. The other summands will contribute 0 when multiplied by the larger power of $(1-q)$ and the limit is taken. For each maximal path we have $\{i+j|(i,j) \in P\} = \{2,\ldots,2n-2\}$ and the number of these paths is $\binom{2n-4}{n-2}$, so
$$(1-q)^{2n-3}c_s(q^{n-1},\ldots,q^{2n-3})= \binom{2n-4}{n-2} \prod_{i=2}^{2n-2} \frac{1-q}{1-q^{2(n-2)+i}} + (1-q)...,$$
where the remaining terms are divisible by $1-q$, hence contribute 0 when the limit is taken.

Now we can proceed to compute $\lim_{q \rightarrow 1} (1-q)^{ N} S_{n,1}(q^1,q^2,\ldots,;q^{n-1})$.
Putting all these together we have that
\begin{align*}
\lim_{q \rightarrow 1}& (1-q)^{\binom{n+1}{2}-1}S_{n,1}(q^1,\ldots,q^{n-2};q^{n-1})  \\
&= \lim_{q\rightarrow 1} \prod_{1\leq i\leq n-2} \frac{1-q}{1-q^i} 
   \prod_{1\leq i<j\leq n-2} \frac{1-q}{1 - q^{i+j}} \frac{1-q}{1-q^{2(n-1)}}\left((1-q)^{2n-3}(c_{1}(u)+c_{0}(u))\right)\\
  &=   \prod_{0\leq i<j\leq n-2} \frac{1}{i+j} \frac{1}{2(n-1)}2\binom{2n-4}{n-2} \prod_{i=2}^{2n-2} \frac{1}{2n-4+i}=  \frac{g_{\delta_{n-2}}}{\binom{n-1}{2}!} \frac{1}{(n-1)}\binom{2n-4}{n-2} \frac{(2n-3)!}{(4n-6)!},
\end{align*}
where $g_{n-2}=\frac{ \binom{n-1}{2}!}{\prod_{0\leq i<j \leq n-2}(i+j)}$ is the number of shifted staircase tableaux of shape $(n-2,\ldots,1)$. After algebraic manipulations we arrive at the desired formula.
\end{proof}

\section{Standard tableaux of shape rectangle truncated by staircase}

We will compute the number of standard tableaux of straight truncated shape $D=n^m \setminus \delta_k$ using propositions \ref{sumofentries}, \ref{limformula} and \ref{d-formula}.

By Proposition \ref{sumofentries} we have that 
\begin{align*}
F_D(q) &= \sum_{T: \sh(T)=D} q^{\sum T[i,j]}\\
&= \sum_{\lambda,\mu|l(\mu)\leq k+1,l(\lambda) \leq n}\sum_{P,\sh(P),\lambda/\mu,Q,\sh(Q)=\lambda} q^{\sum P[i,j] +\sum Q[i,j] -|\lambda| +(n-k)|\mu|} \\
&= \sum_{\lambda,\mu| l(\lambda)\leq n, l(\mu)\leq k+1} s_{\lambda}(1,q,q^2,\ldots,q^{m-1})s_{\lambda/\mu}(q,q^2,\ldots,q^{n-k-1})q^{(n-k)|\mu|},
\end{align*}
which is $D_{n,m,k}(x,z;t)$ for $x=(q,q^2,\ldots,q^{n-k-1})$, $z=(1,q,\ldots,q^{m-1})$ and $t=q^{n-k}$ and from its simplified formula from Proposition \ref{d-formula} and Proposition \ref{limformula} the number of standard tableaux of shape $n^m\setminus \delta_k$ is 
\begin{align}\label{rect_limit}
\lim_{q\rightarrow 1}&(1-q)^{nm-\binom{k+1}{2}}F_D(q)\\
&=\lim_{q\rightarrow 1} \left(\prod_{i=1,j=0}^{n-k-1,m-1} \frac{1-q}{1-q^{i+j}} (1-q)^{m(k+1)-\binom{k+1}{2}}(\sum_{\nu|l(\nu) \leq k+1} s_{\nu}(q^{n-k},q^{n-k+1},\ldots,q^{m-1+n-k}))\right).\notag
\end{align}
We are thus going to compute the last factor. 
\begin{lemma}
Let $p\geq r$ and $N = rp -\binom{r}{2}$. Then for any $s$ we have 
\begin{align*}
\lim_{q\rightarrow 1}(1-q)^N \sum_{\lambda|l(\lambda) \leq r} s_{\lambda}(q^{1+s},\ldots,q^{p+s})=
\frac{g_{(p,p-1,\ldots,p-r+1)}}{N!}\frac{E_1(r,p,s)}{E_1(r,p,0)},
\end{align*}
where $$E_1(r,p,s)=\prod_{r<l<2p-r+2}  \frac{1}{(l+2s)^{r/2}} \prod_{2\leq l\leq r}
\frac{1}{((l+2s)(2p-l+2+2s))^{\lfloor l/2 \rfloor}}$$ for $r$ even and 
$E_1(r,p,s) = \frac{((r-1)/2+s)!}{(p-(r-1)/2+s)!}E_1(r-1,p,s)$ when $r$ is odd and $g_{\lambda}$ is the number of shifted SYTs of shape $\lambda$.
\end{lemma}
\begin{proof}
Consider the Robinson-Schensted-Knuth (RSK) correspondence between SSYTs with no more than $r$ rows filled with $x_1,\ldots,x_p$ and symmetric $p \times p$ integer matrices $A$. The limit on the number of rows translates through Schensted's theorem to the fact that there are no $m+1$ nonzero entries in $A$ with coordinates $(i_1,j_1),\ldots,(i_{r+1},j_{r+1})$, s.t. $i_1<\cdots<i_{r+1}$ and $j_1>\cdots >j_{r+1}$ (i.e. a decreasing subsequence of length $r+1$ in the generalized permutation corresponding to $A$). Let $\mathcal{A}$ be the set of such matrices. Let $\mathcal{A_0}\subset\mathcal{A}$ be the set of $0-1$ matrices satisfying this condition, we will refer to them as allowed configurations. Notice that $A \in \mathcal{A}$ if and only if $B \in \mathcal{B}$, where $B[i,j] = \begin{cases} 1, \text{ if } A[i,j]\neq 0 \\0, \text{ if }A[i,j]=0\end{cases}$.
We thus have that
\begin{align*}
\sum_{\lambda|l(\lambda)\leq r} s_{\lambda}(x_1,\ldots,x_p) &= \sum_{A\in \mathcal{A}} \prod_{i}x_i^{A[i,i]} \prod_{i> j}(x_ix_j)^{A[i,j]}\\
&= \sum_{B \in \mathcal{B}} \prod_{i: B[i,i]=1}(\sum_{a_{i,i}=1}^{\infty}x_i^{a_{i,i}} )\prod_{i>j: B[i,j]=1}(\sum_{a_{i,j}=1}^{\infty} (x_ix_j)^{a_{i,j}})\\
&= \sum_{B \in \mathcal{B}} \prod_{i: B[i,i]=1} \frac{x_i}{1-x_i} \prod_{i>j: B[i,j]=1} \frac{x_ix_j}{1-x_ix_j}.
\end{align*}
Notice that $B$ cannot have more than $N$ nonzero entries on or above the main diagonal. No diagonal $i+j=l$ (i.e.\ the antidiagonals) can have more than $r$ nonzero entries on it because of the longest decreasing subsequence condition. Also if $l\leq r$ or $l> 2p-r+1$, the total number of points on such diagonal are $l-1$ and $2p-l+1$ respectively. Since $B$ is also symmetric the antidiagonals $i+j=l$ will have $r-1$ entries if $l\equiv r-1(\bmod 2)$ and $r$ is odd. Counting the nonzero entries on each antidiagonal on or above the main diagonal gives always exactly $N$ in each case for the parity of $r$ and $p$.
If $B$ has less than $N$ nonzero entries, then
\begin{align*}
\lim_{q \rightarrow 1}(1-q)^N& \prod_{i: B[i,i]=1} \frac{q^{i+1}}{1-q^{i+s}} \prod_{i>j: B[i,j]=1} \frac{q^{i+j+2s}}{1-q^{i+j+2s}} =\\
 \lim_{q \rightarrow 1}(1-q)^{N-|B|>0}& \prod_{i: B[i,i]=1} \frac{q^{i+1}(1-q)}{1-q^{i+s}} \prod_{i>j: B[i,j]=1} \frac{q^{i+j+2s}(1-q)}{1-q^{i+j+2s}} =0,
 \end{align*}
so such $B$s won't contribute to the final answer.
 
Consider now only $B$s with maximal possible number of nonzero entries (i.e.\ $N$), which forces them to have exactly $r$ (or $r-1$) nonzero entries on every diagonal $i+j=l$ for $r<l\leq 2p-r+1$  and all entries in $i+j\leq r$ and $i+j >2p-r+1$. 

If $r$ is even, then there are no entries on the main diagonal when $r<l<2p-r+2$ and so there are $r/2$ terms on each diagonal $i+j=l$. Thus every such $B$ contributes the same factor when evaluated at $x=(q^{1+s},\ldots,q^{p+s})$:
$$E_q(r,p,s):=\prod_{r<l<2p-r+2}  \frac{q^{(l+2s)r/2}}{(1-q^{l+2s})^{r/2}} \prod_{2\leq l \leq r}
\frac{q^{(l+4s +2p-l+2)\lfloor l/2\rfloor}}{((1-q^{l+2s})(1-q^{2p-l+2+2s}))^{\lfloor l/2 \rfloor}}.$$
If $r$ is odd, then the entries on the main diagonal will all be present with the rest being as in the even case with $r-1$, so the contribution is
$$E_q(r,p,s):=\prod_{\frac{r+1}{2}\leq i\leq p-\frac{r+1}{2}+1}\frac{q^{i+s}}{1-q^{i+s}}E_q(r-1,p,s).$$
Let now $M$ be the number of such maximal $B$s in $\mathcal{A}_0$. The final answer after taking the limit is
$M E_1(r,p,s)$, where we have that $E_1(r,p,s) = \lim_{q\rightarrow 1}(1-q)^NE_q(r,p,s)$ as defined in the statement of the lemma.

In order to find $M$ observe that the case of $s=0$ gives
\begin{align}
\lim_{q\rightarrow 1}(1-q)^N \sum_{\lambda|l(\lambda) \leq r} s_{\lambda}(q^{1},\ldots,q^{p}) = M  E_1(r,p,0),
\end{align}
on one hand. On the other hand via the bijection $\phi$ we have that 
\begin{align*}
\sum_{\lambda|l(\lambda \leq m} s_{\lambda}(q^{1},\ldots,q^{n})=\sum_{T} q^{\sum T[i,j]},
\end{align*}
where the sum on the right goes over all shifted plane partitions $T$ of shape $(p,p-1,\ldots,p-r+1)$. Multiplying by $(1-q)^N$ and taking the limit on the right hand side gives us, by the inverse of Proposition~\ref{limformula}, $\frac{1}{N!}$ times the number of standard shifted tableaux of that shape. This number is well known and is $g_{(p,p-1,\ldots,p-r+1)}=\frac{N!}{\prod_u h_u}$, where the product runs over the hooklengths of all boxes on or above the main diagonal of $(p^r,r^{p-r})$. Putting all this together gives
$$M E_1(r,p,0) = \frac{g_{(p,p-1,\ldots,p-r+1)}}{N!}.$$
Solving for $M$ we obtain the final answer:
$$ \frac{g_{(p,p-1,\ldots,p-r+1)}}{N!} \frac{E_1(r,p,s)}{E_1(r,p,0)}.$$
\end{proof}

We can now put all of this together and state
\begin{theorem}
The number of standard tableaux of truncated straight shape $\smash{\underbrace{(n,n,\ldots,n)}_m}\setminus \delta_k$ (assume $n\leq m$), is
\begin{align}
(mn-\binom{k+1}{2})! &\times \frac{f_{(n-k-1)^m}}{(m(n-k-1))!} \times \frac{g_{(m,m-1,\ldots,m-k)}}{((k+1)m-\binom{k+1}{2})!}\frac{E_1(k+1,m,n-k-1)}{E_1(k+1,m,0)}, 
\end{align}
where $$ E_1(r,p,s)=\begin{cases} \prod_{r<l<2p-r+2}  \frac{1}{(l+2s)^{r/2}} \prod_{2\leq l\leq r}
\frac{1}{((l+2s)(2p-l+2+2s))^{\lfloor l/2 \rfloor}}, \; r \text{ even}, \\  
\frac{((r-1)/2+s)!}{(p-(r-1)/2+s)!}E_1(r-1,p,s) , \; r \text{ odd}.\end{cases}$$
\end{theorem}
\begin{minipage}{0.55\textwidth}
Notice that the product $E_1$ is just the product of the sum of the coordinates $i+j+2s$ of the boxes $(i,j)$ in the diagonal strip of width $r$ above the main diagonal as in the picture below. When $r$ is even the crossed diagonal squares $(i,i)$ for $r\leq i \leq p-r+1$ are not included, and when $r$ is odd they are included in the product as $i+s$. 
\end{minipage}
\begin{minipage}{0.40\textwidth}
\usetikzlibrary{decorations.pathreplacing}
\begin{tikzpicture}[scale=0.5]
\path[fill=gray] (0,6) -- (3,6) -- (3,5) -- (4,5)--(4,4)--(5,4)--(5,3)--(6,3)--(6,0)--
(5,0)--(5,1)--(4,1)--(4,2)--(3,2)--(3,3)--(2,3)--(2,4)--(1,4)--(1,5)--(0,5)--cycle;
\draw[step =1](0,0) grid (6,6);

\draw[decorate,decoration={brace, amplitude=4pt}] (6.1,3) -- (6.1,0);
\draw (6.7,1.5) node{$r$};

\draw (0.5,6.3) node{$s+1$};
\draw (5.5,6.3) node{$s+p$};
\draw (1.5,6.3) node{$\cdot$};
\draw (2.5,6.3) node{$\cdot$};
\draw (3.5,6.3) node{$\cdot$};
\draw (4.5,6.3) node{$\cdot$};
\draw (2.5,3.5) node{$\times$};
\draw (3.5,2.5) node{$\times$};

\draw (-1,5.5) node{$s+1$};
\draw (-1,0.5) node{$s+p$};
\draw (-0.5,1.5) node{$\cdot$};
\draw (-0.5,2.5) node{$\cdot$};
\draw (-0.5,3.5) node{$\cdot$};
\draw (-0.5,4.5) node{$\cdot$};

\end{tikzpicture}
\end{minipage}

\section{Truncation by almost squares}

We consider now rectangles truncated by a square without a corner, a case also considered in \cite{Roich2}. We will use the methods developed so far to derive a formula for the volume generating function $F_D(q)$ for these shapes and the number of standard truncated tableaux.

Specifically, let $D=n^m \setminus (k^{k-1},k-1)$ be a rectangle truncated by square without a corner. For example, when $n=10,m=8,k=4$, we  have
\ytableausetup{boxsize=1.5ex} $$D = \ydiagram{6,6,6,7,10,10,10,10}\; .$$ 
\ytableausetup{boxsize=2.5ex}
By symmetry we can assume that $n<m$. Consider the case when $2k \leq n+1$. 
For any plane partition $T$ of shape $D$, let $p=T[k,n-k+1]$ be the value of the entry in that missing square corner. Because of the row and column inequalities, $T$ is in bijection with $T'$ of straight truncated shape $n^m \setminus \delta_{k-1}$, where the $T'[i,j]=T[i,j]$ for $(i,j) \in D$ and $T'[i,j] = T[k,n-k+1]=p$ for the values in the extra boxes. For example,
$$T=\ytableaushort{6654,6543,53322,4322211,3221111,2111111} \qquad \longleftrightarrow \qquad
T'=\ytableaushort{6654{*(gray)2},6543{*(gray)2}{*(gray)2},53322{*(gray)2}{*(gray)2},4322211,3221111,2111111}.$$
Conversely, if $T'$ is a plane partition of truncated shape $n^m\setminus \delta_{k-1}$ whose $n-k+2$nd diagonal $\mu =(T[1,n-k+1],\ldots,T[k,n])$ is of all equal entries, i.e. $\mu =(p^k)$ for some $p$, this forces all entries $T[i,j]=p$ for $i\geq n-k+1$ and $j\leq k$. 

Then the generating function $F_{n^m\setminus (k^{k-1},k-1)}(q)$ is obtained, using equation \ref{fixedmu}, from
\begin{align}\label{Hnmk}
H_{n,m,k}(z;x;t) = \sum_p \sum_{\lambda} s_{\lambda}(z)s_{\lambda/(p^k)}(x)t^p=\prod \frac{1}{1-z_ix_j}\sum_{p}s_{(p^k)}(z)t^p,
\end{align}
by substituting $x=(q,q^2,\ldots,q^{n-k})$, $z=(1,q,\ldots,q^{m-1})$ as in the case of the rectangle truncated by a staircase $n^m \setminus \delta_{k-1}$. For the value of $t$ in this case,  since $\sum T[i,j] = \sum T'[i,j] -p(\binom{k+1}{2}-1)$ to account for the extra squares in $T'$, we have $t^p = q^{|\mu|(n-k+1) -p(\binom{k+1}{2}-1)} =(q^{k(n-k+1)-\binom{k+1}{2}+1})^p$.
Then 
\begin{align}\label{Falmostsq}
F_{n^m\setminus (k^{k-1},k-1)}(q)=H_{n,m,k}(1,q,\ldots,q^{m-1};q,q^2,\ldots,q^{n-k};q^{k(n-k)-\binom{k+1}{2}+1})=\notag \\
\prod_{i=0;j=1}^{m-1;n-k} \frac{1}{1-q^{i+j}}\sum_{p}s_{(p^k)}(1,q,\ldots,q^{m-1})(q^{k(n-k)-\binom{k+1}{2}+1})^p
\end{align}
The number of standard truncated tableaux of this shape will then be given by $$\lim_{q\rightarrow 1}(1-q)^{nm-k^2+1} F_{n^m\setminus (k^{k-1},k-1)}(q).$$

We will now evaluate the sum over $p$ and the relevant limit.
\begin{lemma}\label{lemma:thirdshape}
Let $f_q(v) = v^{\binom{k}{2}}\prod_{i\leq k<j}(v q^{m-i} -q^{m-j})=\sum a_i(q)v^i$, then
\begin{align}\label{lemma2_q}
\sum_p s_{(p^k)}(1,\ldots,q^{m-1})t^p =
\prod\limits_{m-k\leq i<m;0\leq j<m-k}\frac{1}{(q^i-q^j)}
\sum_{i} a_i(q) \frac{1}{1-q^i t}.
\end{align}
If $t=q^s$, then
$$\lim_{q\rightarrow 1} (1-q)^{mk-k^2+1}\sum_p s_{(p^k)}(1,\ldots,q^{m-1})(q^s)^p=
\prod\limits_{i=m-k;j=0}^{m-1;m-k-1}\frac{1}{(j-i)}
\frac{(\binom{k}{2}+s)!(k(m-k))!}{(mk -\binom{k+1}{2}+s+1)!}
$$
\end{lemma}
\begin{proof}
Consider the determinantal formula for the Schur functions, we have that
\begin{align}\label{schurdet}
s_{\lambda}(q^{0},\ldots,q^{m-1})\prod_{0\leq j<i<m}(q^i-q^j)= 
 \det [ (q^i)^{\lambda_j+m-j}]_{i,j} =
\det[(q^{\lambda_j+r-j})^i]_{i,j}=\notag\\
\prod_{0<i<j<m+1} (q^{\lambda_i+m-i}-q^{\lambda_j+m-j})
\end{align}
where in our case $\lambda=(p^k,0^{m-k})$.  we have
\begin{align}\label{detformula}
\prod_{0<i<j<m+1} (q^{\lambda_i+m-i}-q^{\lambda_j+m-j}) =\notag \\
\prod_{0<i<j\leq k}(q^p q^{m-i} -q^p q^{m-j})\prod_{0<i\leq k<j<m+1}(q^p q^{m-i} -q^{m-j})\prod_{k<i<j<m+1}(q^{m-i}-q^{m-j})=\notag \\
(q^p)^{\binom{k}{2}}\prod_{0<i<j\leq k}( q^{m-i} - q^{m-j})\prod_{k<i<j<m+1}(q^{m-i}-q^{m-j})
\prod_{0<i\leq k<j<m+1}(q^p q^{m-i} -q^{m-j})
\end{align}
Let $$A(q)=\prod\limits_{0<i<j\leq k}( q^{m-i} - q^{m-j})\prod\limits_{k<i<j<m+1}(q^{m-i}-q^{m-j}),$$ which does not depend on $p$. 
The last equation, \ref{detformula} is thus equal to $A(q)f_q(q^p)$, where $f_q(v)$ is the polynomial defined in the statement of the lemma.
We have
$$\sum_p f_q(q^p)t^p = \sum_{0\leq i \leq l} a_i(q) \sum_p q^{ip}t^p =
\sum_{0 \leq i \leq l} a_i(q) \frac{1}{1-q^i t},$$
which together with $A(q)$, equation \eqref{schurdet} and some cancellations gives the first part of the lemma.
For the limit we need to find
$$\lim_{q\rightarrow 1}(1-q) \sum_{0 \leq i \leq l} a_i(q) \frac{1}{1-q^i q^s}.$$
This limit is equal to
\begin{align*}
\sum_{0 \leq i \leq l} \frac{a_i(1)}{i+s} = \int_{0}^1 f_1(v)v^{s-1}dv=\\
\int_0^1 v^{\binom{k}{2}+s-1}
\prod_{0<i\leq k<j<m+1}(v 1^{m-i} -1^{m-j})dv=
\int_0^1 v^{\binom{k}{2}+s-1}(v-1)^{k(m-k)}dv
\end{align*}
For the last integral, it could easily be found recursively, e.g. if
$$I(a,b) = \int_0^1 v^a(v-1)^bdv, \text{ then } I(a,b) = -\frac{b}{a+1}I(a+1,b-1),$$
where $I(a,0)=\frac{1}{a+1}$, so $\displaystyle I(a,b) = (-1)^b \frac{b!}{(a+1)\cdots (a+b+1)}.$
Noticing that $\displaystyle \lim_{q\rightarrow 1}\prod\limits_{i,j} \frac{1-q}{q^i-q^j} = \prod \frac{1}{i-j}$ for the remaining $mk-k^2$ factors in equation \eqref{lemma2_q} we obtain the desired limit.   
\end{proof}
For the number of standard tableaux of shape $D$ take the limit $\lim\limits_{q\rightarrow 1}(1-q)^{nm-k^2+1}F_D(q)$ using the formula for $F_D(q)$ from \eqref{Falmostsq} and applying the second part of the Lemma with $s=k(n-k)-\binom{k+1}{2}+1$. We can simplify the products by viewing them as products of hook-lengths and invoking the hook-length formula for the respective shape. 

\begin{theorem}
The generating function for truncated tableaux of shape $D=n^m\setminus (k^{k-1},k-1)$, where $n\geq 2k-1$, is
$$F_D(q) =\sum_r  a_r(q) \frac{1}{1-q^{r+kn-k^2-\binom{k+1}{2}+1}}
\prod\limits_{m-k\leq i<m;0\leq j<m-k}\frac{1}{(q^i-q^j)}\prod_{i=0;j=1}^{m-1;n-k} \frac{1}{1-q^{i+j}}.$$
The number of standard truncated tableaux of that shape is
$$\frac{(nm-k^2+1)!f_{(m^{n-k})}}{(m(n-k))!}
\frac{f_{((m-k)^{k})}}{((m-k)k)!}\frac{(kn-k^2)!(k(m-k))!}{(mk +kn-2k^2+1)!}.$$
\end{theorem}
After cancellations and regrouping of the factorials we obtain the binomial coefficient version of the formula in Theorem \ref{thm3} from the Introduction.

\section{3-dimensional Young diagrams and lozenge tilings}

In this section we will mention some connections with 3d diagrams/stepped surfaces/lozenge tilings, which are of interest as models in statistical mechanics. For an introduction to the subject see, e.g. \cite{Kenyon}.

Plane partition $T$ of shape $D$ can be thought as a 3D diagram of unit cubes, see Figure~\ref{fig:3dpp}, where the base in the $xy$ plane is $D$ and on top of each square $[i,j]\in D$ there are $T[i,j]$ unit cubes. Its volume is the number of cubes, i.e. $\sum T[i,j]$, so our generating function $F_D(q)$ is the same as the volume generating function for plane partitions of base $D$:
$$F_D(q) = \sum_{T} q^{Vol(T)}.$$

 We usually draw 3D Young diagrams as their projection on the $x+y+z=0$ plane, where each side of the unit cube becomes a rhombus. The rhombi do not overlap, so  the projection of $T$ becomes a lozenge (pair of equilateral triangles) tiling of the triangular grid with certain boundary.  

\begin{figure}\label{fig:3dpp}
\begin{center} \begin{tabular}{ccc} 
\begin{minipage}{2in} 
\includegraphics[width=1.7in]{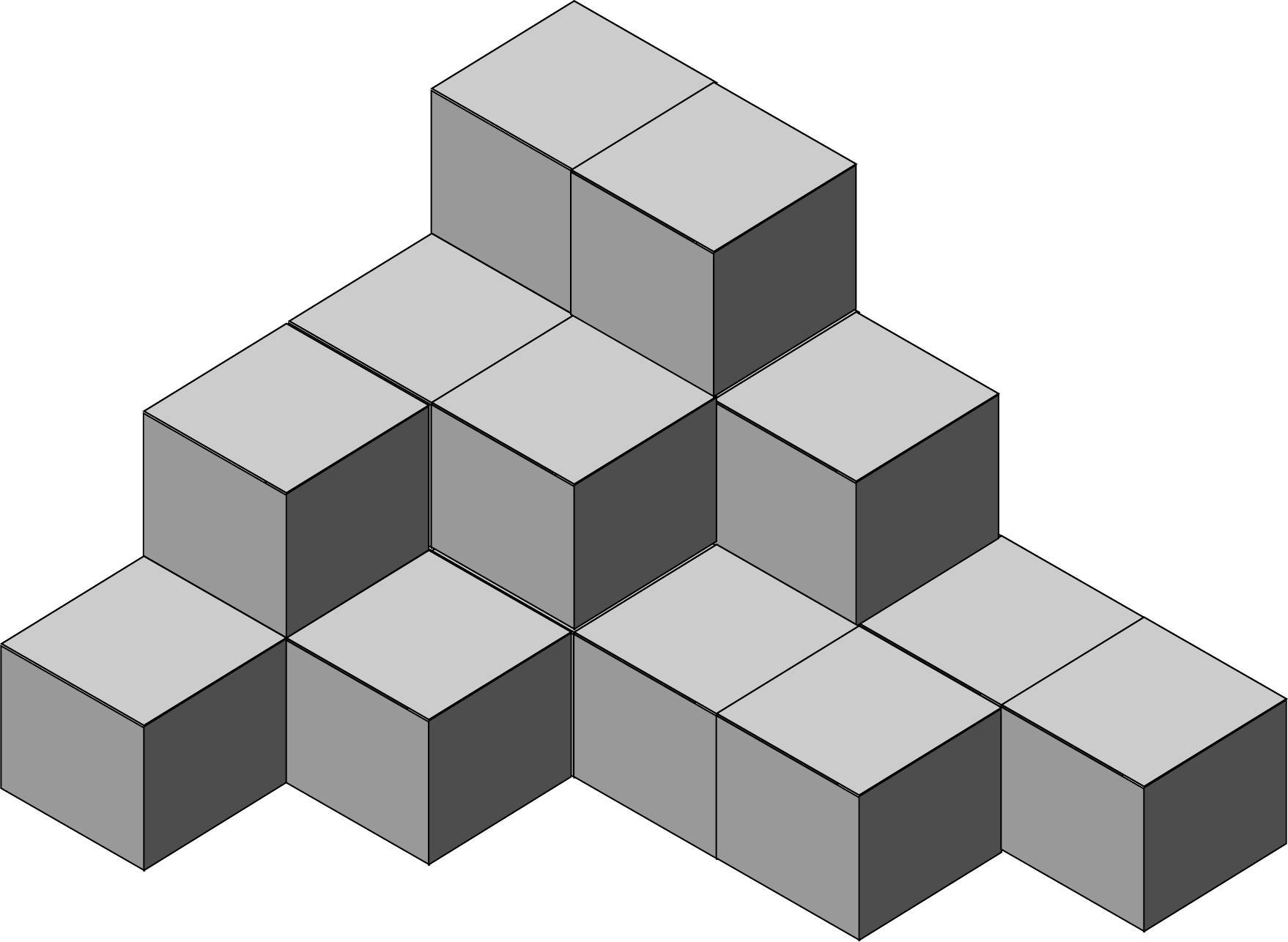}
\end{minipage}&$\quad$&\begin{minipage}{2.5in}\includegraphics[width=2.5in]{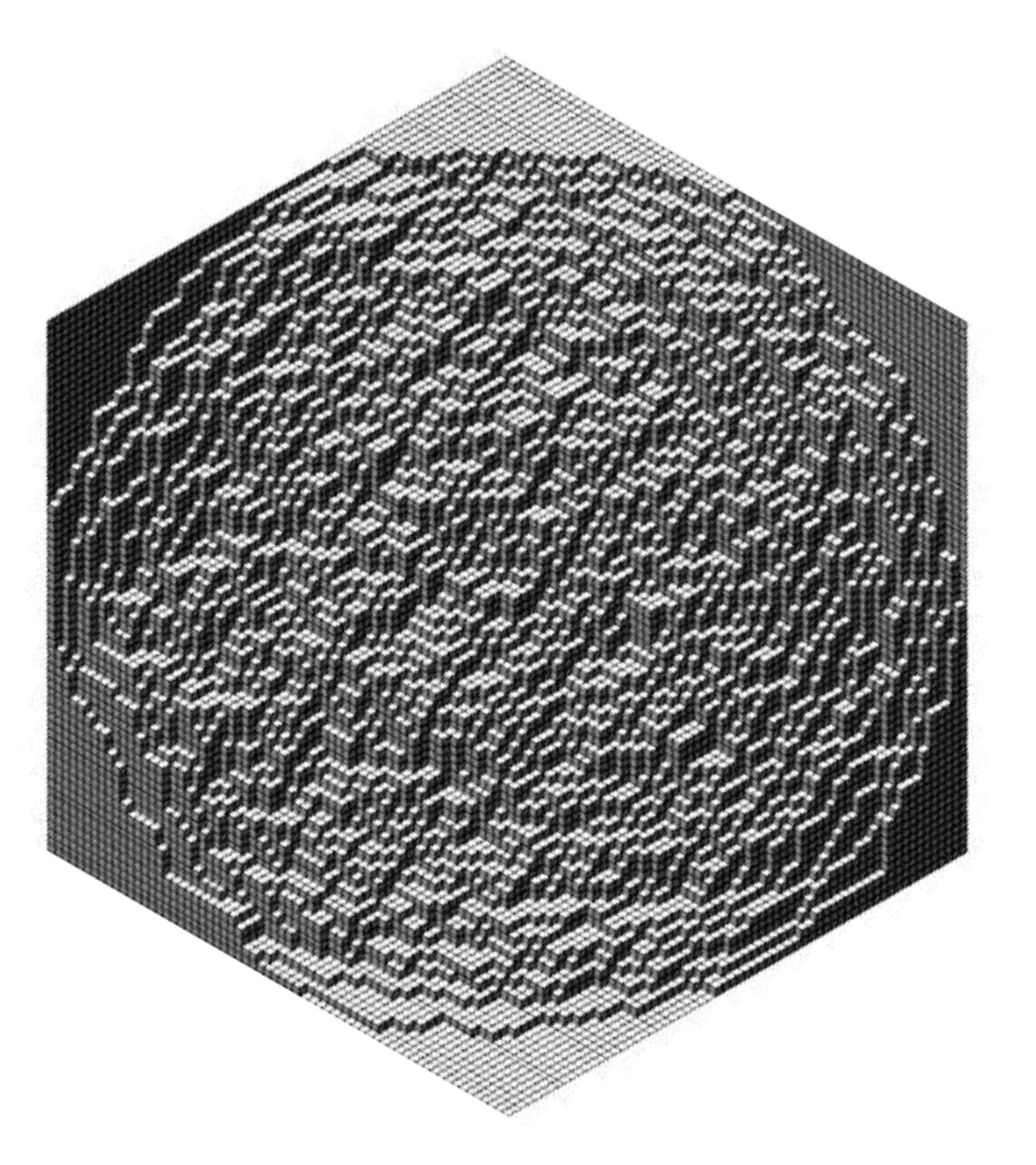} \end{minipage}
\end{tabular}\end{center}

\caption{The picture on the left above show the 3d Young diagram corresponding to $T=\ytableaushort{33211,2211,21,1}$ of shape $(4,3,3,1)$. If $D=n^n$ and $T[1,1]=n$, as $n\rightarrow \infty$, such plane partitions approximate stepped surfaces, as the picture on the right shows. }
\end{figure}

Using the methods developed so far, in particular equation \eqref{fixedmu}, we can easily derive volume generating functions for 3d Young diagrams with certain restrictions. Alternatively, these would correspond to lozenge tilings with free boundary. Namely, truncated plane partitions of shape $n^m \setminus \delta_k$, when drawn in 3 dimensions, are lozenge tilings of a hexagon with sides $m,a,n,m-k,a,n-k$, where $a \rightarrow \infty$ (i.e. height is unbounded). Here $k$ of the lozenges on the right side will be ``sticking out'', that is $k$ of the triangles on the right boundary will be unmatched thus corresponding to a semi-free boundary condition, see Figure~\ref{lozenges}.  

\begin{figure}\label{lozenges}
\begin{center}
\begin{tabular}{p{1.8in}p{1.8in}p{1.8in}}
\includegraphics[width=1.5in]{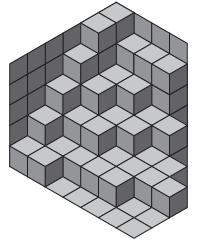} &
\includegraphics[width=1.5in]{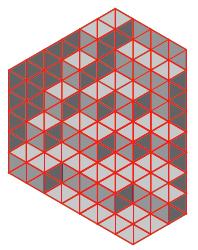}&
\includegraphics[width=1.5in]{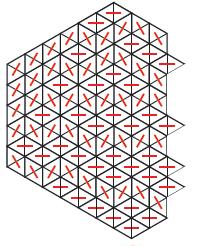} 
\end{tabular}
\end{center}

\caption{ The picture on the left represents the 3d diagram of a truncated plane partition with $n=8$, $m=6$, $a=5$ and $k=3$. The middle picture shows the underlying triangular grid and the picture on the right shows the corresponding matching of the triangles in the lozenges, where each lozenge is the side of cube from the first picture. }
\end{figure}

Using the results we have obtained so far about $F_D(q)$ for various shapes $D$ we can state the following facts about the $q^{vol}$ distribution of the truncated 3d diagrams. Setting $D=n^n$ and $\mu=(b)$ in the formulae for $F_D(q)$ from the section on truncated rectangles we get the following 2 results.

\begin{prop}
The average value under the volume statistic of boxed plane partitions of base $(n^n)$ and whose corner is at a fixed value $b$, i.e. $T[1,n]=b$, is given by
$$\frac{\sum_{T:\sh(T)=n^n; T_{1n}=b} q^{\sum T_{ij} }}{\sum_{T:\sh(T)=n^n}q^{\sum T_{ij}}} = \prod_{i=1}^{n}(1-q^{n-1+i})\left[ n+b-1 \atop b \right]_q.$$
\end{prop} 

More generally, the volume generating function for plane partitions on $(n^n)$, whose $n-k$'th diagonal is fixed at a certain value $\mu$, i.e. $(T[1,n-k+1],T[2,n-k+2],\ldots)=\mu$, is derived as
$$\sum s_{\lambda}(x)s_{\lambda/\mu}(y)s_{\mu}(z) =
\prod \frac{1}{1-x_iy_j } s_{\mu}(x)s_{\mu}(z)$$
for $x=(1,\ldots,q^{n-1})$, $y=(q,q^2,\ldots,q^{n-k})$ and $z=(q^{n-k+1},\ldots,q^n)$.
Notice that $\mu$ encodes the positions of the unmatched triangles (``sticking out'' lozenges) in the lozenge tiling interpretation. 
Thus we have that
\begin{multline*}
\sum_{T:\sh(T)=n^n; (T[1,n-k+1],\ldots,T[k,n])=\mu} q^{\sum T[i,j]} =
q^{(n-k+1)|\mu|} \frac{ \prod_{0<i<j<n+1}(q^{\mu_i+n-i} -q^{\mu_j+n-j})}{\prod_{0<i<j<n+1}(q^j-q^i)}\times \\ \frac{ \prod_{0<i<j<k+1}(q^{\mu_i+k-i}-q^{\mu_j+k-j})}{\prod_{0\leq i<j\leq k}(q^j-q^i)}\prod_{0<i<n+1,0<j<n-k+1} \frac{1}{1-q^{i+j-1} }.
\end{multline*}

To obtain the average distribution we can use the determinantal formula for the sum of Schur functions over partitions of restricted length [Ron King, unpublished notes]:
$$\sum_{\lambda, l(\lambda)\leq m} s_{\lambda}(x_1,\ldots,x_n) = \frac{1}{\prod_{i<j}(x_i-x_j)(1-x_ix_j)\prod_i (1-x_i)} \det[ x_i^{n-j}-(-1)^m\chi(j>m)x_i^{n-m+j-1}]_{i,j=1}^n.$$

\ytableausetup{boxsize=0.4em} 
\parbox{0.2\textwidth}{\ydiagram{6,7,8,9,10,11,12,13,13,13,13,13} * [*(blue)]{6+1,7+1,8+1,9+1,10+1,11+1,12+1} }\parbox{0.8\textwidth}{
\begin{prop}
The distribution for truncated plane partitions with base $n^m$ and fixed rightmost diagonal $=\mu$ is
$$\frac{ q^{|\mu|(n-k)}\prod_i(1-q^{n-k+i-1})\prod_{i<j}(q^{\mu_i+m-i}-q^{\mu_j+m-j})(1-q^{2n-2k+i+j-2}) }
{\det[q^{(n-k+i-1)(m-j)}-(-1)^{k+1}\chi(j>k+1)q^{(n-k+i-1)(m-k+j-2)}]_{i,j}}$$
\end{prop}} 

We can also easily obtain the $q^{vol}$ distribution of ordinary plane partitions with base $n\times n$ and fixed $n-k$th diagonal $\mu$ as in the picture
$\qquad \ydiagram{13,13,13,13,13,13,13,13,13,13,13,13} * [*(blue)]{6+1,7+1,8+1,9+1,10+1,11+1,12+1}$. 
\begin{prop}The average wrt $q^{Vol}$ distribution of PP of shape $n\times n$ with a fixed $n-k$'th diagonal $(T[1,n-k+1],T[2,n-k+2],\ldots)=\mu$ is
\begin{multline*}
q^{(n-k+1)|\mu|} \prod_{0<i<j<n+1}\frac{ (q^{\mu_i+n-i} -q^{\mu_j+n-j})}{(q^j-q^i)}\times \\ \prod_{0<i<j<k+1} \frac{ (q^{\mu_i+k-i}-q^{\mu_j+k-j})}{(q^j-q^i)}\prod_{0<i<n+1,n-k<j<n+1} \frac{1}{1-q^{i+j-1} }.
\end{multline*}
\end{prop}


\arxiv{   }

\adv{

}


\begin{bibdiv}
\begin{biblist}

\bib{Roich}{article}{
 author = {Adin, R.},
 author = {Roichman, Y},
 title = {Triangle-Free Triangulations, Hyperplane Arrangements and Shifted Tableaux},
 journal={},
 review = {arXiv:1009.2628v1},
}

\bib{Roich2}{article}{
author = {Adin, R.},
author = {King, R.},
author= {Roichman, Y.},
title={Enumeration of standard Young tableaux of certain truncated shapes},
journal={Electron. J. Combin.},
   volume={18},
   date={2011},
}

\bib{Kenyon}{article}{
author={Kenyon, Richard},
title={Lectures on dimers},
journal={Statistical mechanics, IAS/Park City Math. Ser.}, 
volume={16},
publisher={ Amer. Math. Soc.},
place={ Providence, RI},
year={2009},
}

\bib{Mac}{book}{
   author={Macdonald, I. G.},
   title={Symmetric functions and Hall polynomials},
   publisher={The Clarendon Press Oxford University Press},
   place={New York},
   date={1995},
}
\bib{EC2}{book}{
   author={Stanley, Richard P.},
   title={Enumerative combinatorics. Vol. 2},
   publisher={Cambridge University Press},
   place={Cambridge},
   date={1999},
}   


\end{biblist}
\end{bibdiv}

\begin{thebibliography}{9}
\bibitem{Roich}
R. Adin,
 Y. Roichman,
 Triangle-Free Triangulations, Hyperplane Arrangements and Shifted Tableaux,
  preprint (2009) arXiv:1009.2628v1.


\bibitem{Roich2}
R. Adin,
R. King,
Y. Roichman,
Enumeration of standard Young tableaux of certain truncated shapes,
Electron. J. Combin. 18
   ( 2011 ).


\bibitem{Kenyon}
Richard Kenyon, Lectures on dimers,
Statistical mechanics, IAS/Park City Math. Ser. 16
16 (2009), Amer. Math. Soc., Providence, RI.


\bibitem{Mac}
  I. G. Macdonald, 
   Symmetric functions and Hall polynomials,
   The Clarendon Press Oxford University Press,
  New York,
   1995.

\bibitem{EC2}
   Richard P. Stanley,
   Enumerative combinatorics. Vol. 2,
  Cambridge University Press,
   Cambridge,
   1999,  

\end{thebibliography}
\end{document}